\documentclass{article}
\usepackage{amsmath,amssymb,amsthm}
\usepackage{url}
\usepackage{hyperref}

\let\norm\|
\def\|{\mkern1.5mu{|}\mkern1.5mu} 
\DeclareMathAlphabet{\mathbfit}{OML}{cmm}{b}{it}
\let\sm\smallsetminus
\def\rr{\mathbfit r}
\newcommand\eqdef{\stackrel{\textrm{def}}{=}}
\newcommand\tsum{{\textstyle\sum}}
\newcommand\restr{\mathord\upharpoonright}
\let\osharp=\sharp
\def\sharp{\mathchoice{\raisebox{1.2pt}{$\osharp\mkern 1mu$}}%
                      {\raisebox{1.2pt}{$\osharp\mkern 1mu$}}%
                      {\raisebox{1pt}{$\mkern 1mu\scriptstyle\osharp$}}%
                      {\raisebox{0.5pt}{$\scriptscriptstyle\osharp$}}}
\newcommand\bref[1]{#1} 
\newcommand\blabel[1]{} 

\newtheorem{theorem}{Theorem}
\newtheorem{lemma}[theorem]{Lemma}
\theoremstyle{definition}
\newtheorem{definition}{Definition}
\newtheorem*{condition}{Conditions}

\title{On the Supremum of Singleton Ratios in Submodular Functions}
\author{L. Csirmaz}
\date{\small R\'enyi Institute, Budapest, and UTIA, Prague}

\begin{document}

\maketitle

\begin{abstract}
\noindent
Let $N$ be a finite set of cardinality $n$, and $a\in N$. A submodular
function $f$ on $N$ with $f(a)=1$ is defined to be $a$-reduced if, for any
decomposition $f=g+h$ into submodular functions where $h$ does not depend on
$a$, it follows that $h$ is identically zero. The maximal possible value of
$f$ on the remaining singletons defines a quantity $\lambda$ that
characterizes the degree to which one variable can constrain the value of
another; geometrically, it also limits the possible elongation of the
associated submodular base polytope. We construct an example demonstrating
that $\lambda$ can be as large as $\Omega(n/\log n)$. Furthermore, we
establish a doubly exponential upper bound on $\lambda$. The problem of
narrowing the gap between these bounds remains open.
\end{abstract}

\section{Introduction}

Submodular functions model the law of diminishing marginal returns, where
the incremental value of an item decreases as the set of items grows, see
\cite{balcan2018} or \cite{vives}. Formally, if $A\subseteq B$, then
$f(A\cup\{a\})-f(A) \ge f(B\cup\{a\})-f(B)$. For the purposes of this paper,
we focus on the class of monotone, pointed submodular functions, often
referred to as \emph{polymatroid rank functions} \cite{fujishige}. These
functions are crucial for optimizing subset selection, such as portfolio
diversification, facility location, and auction bidding \cite{balcan2018}.
Finding bounds for the values of submodular functions, as well as
understanding their structural properties, intersects multiple fields of
applied and pure mathematics, including lattice theory, probability,
combinatorial optimization, and machine learning, see \cite{bach2013}. Such
bounds are highly relevant across several domains because submodularity is
the mathematical engine behind modeling ``synergy'' or ``dependence.''
Determining exact values is a deep, and almost always an unsolved problem in
combinatorial optimization and polyhedral combinatorics \cite{manyrays}.
Estimating parameters of submodular functions goes back to the seminal paper
of J.~Edmonds \cite{edmonds}. He introduced the submodular cone and
connected it to polyhedral geometry via polymatroids and generalized
permutahedra. He was the first to highlight the extreme difficulty of
characterizing the indecomposable building blocks (the extreme rays) of this
cone.

Motivated by results in secret sharing \cite{beimel}, where properties of
submodular functions are used to establish bounds on the efficiency of share
distribution, we extend this line of inquiry to investigate whether an
``inherent'' bound exists for the ratio $f(b)/f(a)$ of a submodular function
$f$ where $a$, $b$ are singletons and $f(a)>0$. Since the sum of submodular
functions is submodular, ``inherent'' means that the bound applies only
after subtracting a maximal component that does not depend on $a$; that is,
when the submodular function is \emph{$a$-reduced}. Translated to the
polyhedral representation of submodular functions \cite{edmonds, manyrays},
this problem asks how elongated the base polytope of an $a$-reduced
submodular function can be. More precisely, what is the maximal possible
ratio of the length of the edges adjacent to a common vertex of this
polytope? We establish that the ratio of the edges of its surrounding box 
can be at least $n/\log n$, proving that
the elongation grows with $n$, the size of the ground set. While theoretical
constraints place a doubly exponential upper bound of $2^{2^n}$ on this
ratio, we conjecture that the true growth is significantly smaller,
opening a new direction in the study of submodular geometry.
Possible applications of these bounds are:
\begin{enumerate}
\item \emph{Combinatorial Optimization and Machine Learning.} Submodular
functions are ubiquitous in algorithms, such as graph cuts, facility
location, sensor placement, etc. In algorithm design one can decompose a
submodular optimization problem into a linear combination of simpler
problems. These bounds limit the ratio of the complexity of these
subproblems, and optionally help to prove mathematical lower bounds on how fast
an algorithm can optimize submodular functions over large data sets.
\item \emph{Game Theory and Economics.}
In cooperative game theory, an extremal submodular function represents a
fundamental, irreducible market dynamic of game structure. Knowing these
bounds allows economists to estimate the necessary resources, based on one
item, in the absolute worst-case scenarios for the core of a cooperative game.

\item \emph{Representability in Neural Networks.}
The recent work \cite{piecewise} connects the extension complexity of
polytopes--a measure of how efficiently a shape can be represented--to the
minimum size of ReLU or Maxout networks. Our lower bound of $n/\log n$ on the
``elongatedness'' of the base polytope suggests that as the ground set grows,
the complexity of the neural architectures required to optimize these
functions must also scale. This provides a geometric foundation for
understanding the depth and width requirements for machine learning models
attempting to learn submodular functions \cite{balcan2018}.
\end{enumerate}

The rest of this paper is organized as follows. Section
\ref{sec:background} provides the necessary background on submodular
functions. In Section \ref{sec:main}, we formally define $a$-reduction and prove
our main bound on the ratio $f(b)/f(a)$. Finally, Section
\ref{sec:conclusion} explores implications and lists open problems.

\section{Background}\label{sec:background}

All sets in this paper are finite. Capital letters $A$, $B$, $I$, $K$, etc.,
denote subsets of the fixed \emph{ground set $N$}, also called \emph{base}.
Elements of the ground set are denoted by lowercase letters such as $a$,
$b$, $i$, $j$; they are also called \emph{variables}. The union sign is
frequently omitted, as well as the curly brackets around singletons; thus
$Aa$ denotes the set $A\cup\{a\}$. The ground set $N$ is assumed to have at
least two elements, and $n=|N|$ denotes its cardinality.

We consider only functions that assign real numbers to subsets of
$N$; we say that $f$ is \emph{defined on $N$}, rather than $f$ is defined on the
subsets of $N$. The function $f$ is \emph{pointed} if
$f(\emptyset)=0$, and it is \emph{monotone} if $A\subseteq B$ implies
$f(A)\le f(B)$. The function $f$ on $N$ is \emph{submodular} if the inequality
\begin{equation}\label{eq:submod}
   f(A)+f(B) \ge f(A\cap B)+f(A\cup B)
\end{equation}
holds for arbitrary subsets $A$, $B\subseteq N$; it is \emph{supermodular} if
(\ref{eq:submod}) holds with the inequality sign reversed; and is
\emph{modular} if it is both submodular and supermodular, that is,
(\ref{eq:submod}) holds with equality for all subsets $A$ and $B$. A modular
function $\rr$ can be written as
$$
  \rr:  A \mapsto w_\emptyset+\tsum \{ w_i: i\in A \}, ~~~~ A\subseteq N,
$$
where $w_\emptyset$ and $w_i$ for $i\in N$ are some real numbers. The modular
function $\rr$ is pointed iff $w_\emptyset=0$, and is monotone iff $w_i\ge 0$ for
all $i\in N$. If not stated otherwise, functions on $N$
are assumed to be pointed and monotone. We remark that for an arbitrary set
function $f$ defined on the subsets of $N$, there is a modular function $\rr$
such that $f+\rr$ is both pointed and modular; that is, $f$ is both pointed
and monotone up to a ``modular shift.'' Submodular functions that are also
pointed and monotone are often referred to as \emph{polymatroids} or
\emph{polymatroidal rank functions}, see, e.g., \cite{fujishige}. An
important example of a polymatroidal rank function is the Shannon entropy of
the marginals of $n$ jointly distributed random variables; the commonly used
name ``variable'' for elements of $N$ originates from this example.

For pointed and monotone functions, submodularity is equivalent to the so-called
\emph{diminishing marginal returns} property, expressed as
\begin{equation}\label{eq:1}
  A\subseteq B ~~\mbox{ implies } ~~ f(aA)-f(A) \ge f(aB)-f(B).
\end{equation}
In an economic model, $f(A)$ can denote value of the portfolio containing
the collection of assets in $A$; the difference $f(aA)-f(A)$ is interpreted
as the additional (marginal) value when the asset $a$ is added to $A$.
Formula (\ref{eq:1}) expresses the natural expectation that adding the same
asset to a larger portfolio yields smaller marginal returns.

Notions like conditional entropy and mutual information from Information
Theory prove useful in the study of polymatroids. These notions are
formally extended to arbitrary set functions and will be used as
abbreviations:
\begin{align*}
   f(A\|B) &\eqdef f(AB)-f(B), \\
   f(A,B)  &\eqdef f(A)+f(B)-f(AB), \mbox{ and}\\
   f(A,B\|C) &\eqdef f(AC)+f(BC)-f(ABC)-f(C).
\end{align*}
Using this notation, the marginal returns $f(aA)-f(A)$ can be written as
$f(a\|A)$, while submodularity is the non-negativity of the expression
$f(A,B\|C)$. Polymatroidal rank functions are axiomatized by the so-called
\emph{basic Shannon inequalities}, listed in (\bref{B1}) and (\bref{B2})
below; see, e.g., \cite{yeung-book}:
\begin{itemize}\setlength\itemsep{3pt plus 1pt minus 2pt}
\item[(B1)]\blabel{B1}$f(\emptyset)=0$ and $f(i\|N\sm i)\ge 0$ for all $i\in N$;
\item[(B2)]\blabel{B2}$f(a,b\|K)\ge 0$ for all $K\subset N$ and different
$a,b\in N\sm K$, including $K=\emptyset$.
\end{itemize}
The set of conditions in (\bref{B1}) and (\bref{B2}) is minimal in the sense
that none of them is a consequence of the others \cite{yeung-book}. There
are $n+1$ constraints in (\bref{B1}), and ${n\choose 2}2^{n-2}$ constraints
in (\bref{B2}) for the $2^n$ possible values of the function $f$.
Non-negative linear (conic) combinations of polymatroids are polymatroids;
this follows from the fact that these constraints are linear. Consequently,
the collection of polymatroids on a fixed base set $N$ forms a polyhedral
cone \cite{ziegler}. Polymatroids on the extremal rays of this cone are
called \emph{extremal}. Extremal polymatroids are also characterized by the
property that they can only be decomposed in a trivial way. Namely, if
$f=g+h$, then both $g$ and $h$ are non-negative multiples of $f$. Another
characterization is that an extremal polymatroid satisfies $2^n-1$ linearly
independent constraints from (\bref{B1}) and (\bref{B2}) with equality
(including $f(\emptyset)=0$); these constraints determine the polymatroid up
to a multiplicative factor \cite{impossible}.

For a polymatroid $f$ on $N$ the \emph{base polytope} associated with $f$ is
the collection of those points $\mathbf x=\langle x_i:i\in N\rangle$ in the
$n$-dimensional Euclidean space that satisfy the conditions
\begin{align}
   & \tsum \{x_i: i\in A\} \le f(A) ~~~\mbox{ for all } A\subset N, \mbox{
and}\notag\\
   & \tsum\{ x_i: i\in N\} = f(N),\label{eq:base2}
\end{align}
see \cite{manyrays}. Points of the base polytope have non-negative
coordinates; thus, it is part of the $n$-dimensional rectangular box
$$
  \mathbb B_f =  \{\mathbf x: 0\le x_i\le f(i) : i\in N\}.
$$
Since it satisfies the equality constraint (\ref{eq:base2}), the base
polytope has dimension at most $n-1$. A characterizing property of the base
polytope is that all of its edges are parallel to $\mathbf e_i-\mathbf e_j$
where $\mathbf e_i$ are the unit coordinate vectors, see
\cite{edmonds,manyrays}.

\section{Main Result}\label{sec:main}

Let $N$ be a finite set with cardinality $n>1$, and $a\in N$ be a fixed
element. Our main goal is to estimate the degree to which the value $f(a)$
of the polymatroid $f$ can constrain the values $f(b)$ for other variables
$b\in N$, namely estimating the amount
$$
   C_{f,a} \eqdef \max_{b\in N}~  f(b)/f(a),
$$
when $f(a)$ differs from zero. Clearly, this value can be pumped up by
adding a polymatroid to $f$ that takes zero at $a$ and some large value at
$b$. Polymatroids that have no such an additive component are called
\emph{$a$-reduced}.

\begin{definition}
(a) The polymatroid $h$ \emph{does not depend on $a$} if $h(aA)=h(A)$ for 
all $A\subseteq N\sm a$. 

(b) A polymatroid $f$ is \emph{$a$-reduced}, if for any decomposition
$f=g+h$ to the sum of two polymatroids so that $h$ does not depend on $a$,
it follows that $h$ is identically zero.
\end{definition}

It is easy to see that $h$ does not depend on $a$ if and only if $h(a)=0$.
Extremal polymatroids with $f(a)>0$ are $a$-reduced (as the components of
any decomposition are multiples of $f$), while a typical $a$-reduced
polymatroid is not extremal. The quantity that characterizes the degree to
which $f(a)$ can constraint the value of other variables can be defined as
follows.
\begin{definition}
Suppose the base set $N$ has $n$ elements. Define $\lambda_n$ as
\begin{equation}\label{eq:lambda}
   \lambda_n \eqdef \sup \{ C_{f,a}: f(a)>0 \mbox{ and $f$ is an 
$a$-reduced polymatroid}
\}.
\end{equation}
\end{definition}

Note that by normalizing the polymatroid such that $f(a)=1$, finding the
supremum of the ratio $f(b)/f(a)$ is strictly equivalent to maximizing the
value of the remaining singletons $f(b)$, which motivates our study of the
quantity $\lambda_n$.

The rest of this section is devoted to the proof of the following theorem.
\begin{theorem}\label{thm:main}
$\displaystyle \frac{n}{2\log_2 n} \le \lambda_n < 2^{2^n}$.
\end{theorem}

\begin{proof}

First we prove that $\lambda_n$ is actually a maximum, and it is taken by an
extremal polymatroid. As discussed in Section \ref{sec:background}, extremal
polymatroids are on the extremal rays of a polyhedral cone, see
\cite{impossible,ziegler}, and there are finitely many such extremal rays.
Let $e_1,\dots,e_k$,$ e_{k+1},\dots, e_\ell$ be polymatroids on these rays
normalized so that $e_i(a)=1$ if $i\le k$, and $e_i(a)=0$ otherwise. Any
polymatroid on $N$ is a non-negative linear combination of these extremal
ones. Let $f$ be an $a$-reduced polymatroid with $f(a)$ positive; without
loss of generality we may assume that $f(a)=1$. Write $f$ as a linear
combination of the extremal polymatroids with non-negative coefficients
$\mu_i$:
$$
   f = \big(\mu_1e_1+\cdots+\mu_k e_k\big) +\big(\mu_{k+1}e_{k+1}+
     \cdots+\mu_\ell e_\ell\big).
$$
Denoting the second term by $h$, we have $h(a)=0$. Since $f$ is $a$-reduced,
it follows that $h$ is identically zero, therefore
\begin{equation}\label{eq:p1}
   f = \mu_1e_1+\cdots+\mu_k e_k.
\end{equation}
Since $f(a)=e_i(a)=1$, the sum of $\mu_1+\cdots+\mu_k$ is exactly
$1$, thus (\ref{eq:p1}) is a convex combination. It means that $C_{f,a}$ is
upper bounded by the maximum of $C_{e_i,a}$, proving that $\lambda_n$ is
taken by one of the extremal polymatroids.

To estimate $C_{e,a}$ for an extremal polymatroid $e$ we use the
characterization that $e$ satisfies $D=2^n-1$ linearly independent
constraints for the $D+1=2^n$ linear variables from (\bref{B1}) and
(\bref{B2}), see \cite{impossible,manyrays}. Each of the constraints
contains at most four non-zero entries from $(1,1,-1,-1)$, forming a
$(D+1)\times D$ matrix $M$. The ratio $f(b)/f(a)$ can be computed as the
ratio of the determinants of two $D\times D$ submatrices of $M$. Since, by
assumption, $f(a)>0$, the denominator is at least $1$ (a non-zero determinant of a
matrix with integer entries). Consequently, an upper bound on the
determinants of the $D\times D$ submatrices of $M$ gives an upper bound on
$C_{e,a}$.

By Hadamard's inequality, the determinant of a $D\times D$ matrix $A$ is
bounded by the product of the $L_2$-lengths of its rows $r_i$. Since each
row contains at most four $\pm1$, the length of $r_i$ is at most $2$, thus
$$
   |\det(A)| \le \prod_{i=1}^D \norm r_i\norm_2 \le 2^D=2^{2^n-1},
$$
proving the upper bound $\lambda_n< 2^{2^n}$ of the theorem.

\medskip

In the second part we will need a lemma which follows from the reasoning
used above.
\begin{lemma}\label{lemma:red}
Every polymatroid $f$ with $f(a)>0$ has a decompositions $f=g+h$ such
that $g$ is $a$-reduced, and $h$ does not depend on $a$.
\end{lemma}
\begin{proof}
Take the decomposition $f=g+h$ in which $h$ does not depend on $a$ and the
value $h(N)$ is maximal. Continuity and boundedness imply that this maximum
is actually taken. Clearly, in this case the component $g$ is $a$-reduced.
\end{proof}

The lower bound is obtained by exhibiting a set of conditions so that
\begin{itemize}
\item[(a)] if $f$ satisfies all conditions, $f=g+h$ such that $g$ is
$a$-reduced, then $g$ also satisfies all conditions;
\item[(b)] if a polymatroid $f$ satisfies all conditions, then $C_{f,a}\ge
n/(2\log_2 n)$;
\item[(c)] some polymatroid on an $n$ element ground set satisfies all conditions.
\end{itemize}
By (c), there is a polymatroid satisfying these conditions. By Lemma
\ref{lemma:red}, it has an $a$-reduced component, and that component,
denoted by $f$, also 
satisfies these conditions by (a). Finally, (b) provides the required lower bound
on $C_{f,a}$, consequently on $\lambda_n$. Let us see the details.

Let $k\ge 2$ be an integer. The ground set $N$ is the disjoint union $\{a\}
\cup X\cup Y$, where $X$ has $k$ elements, and $Y$ has $2^k$ elements, thus
$n=1+k+2^k$. Subsets of $X$ are arranged so that
$$ X_0=X, ~~\dots, ~~  X_{2^k-1}=\emptyset, $$
furthermore, if $i<j$, then $X_i$ is \emph{not} a subset of $X_j$. Elements
of $Y$ are indexed from $0$ to $2^k-1$, $Y_0=\{y_0\}$ and $Y_i=Y_{i-1}\cup
\{y_i\}$ for $i\ge 1$.

Let $f$ be a polymatroid on $N$ so that $f(a)>0$. Without loss of generality
we may assume $f(a)=1$. For each subset
$A\subseteq N\sm a$ the axioms in (\bref{B1}) and (\bref{B2}) imply
that
\begin{equation}\label{eq:eq}
    f(A) \le f(aA) \le f(A)+f(a) = f(A)+1.
\end{equation}
We write $A\in\mathcal L$ if $f(aA)$ should equal its lower bound, and
$A\in\mathcal U$ if $f(aA)$ should equal its upper bound.
\begin{condition}
The following conditions must hold:
\begin{itemize}
\item[(i)] $X_iy_i\in \mathcal L$ for all $0\le i\le 2^k-1$, and
\item[(ii)] $X_iY_{i-1}\in\mathcal U$ for all $1\le i\le 2^k-1$.
\end{itemize}
\end{condition}

\begin{proof}[Proof of property (a)]
If $f=g+h$ and $f(aA)-f(A)$ is extremal, then so are $g(aA)-g(A)$ and
$h(aA)-h(A)$. Therefore, if (i) and (ii) hold for $f$, then they also hold
of $g$.
\end{proof}

\begin{lemma}\label{lemma:3}
Under these conditions, $ f(X,y_i\|Y_{i-1})\ge 1$.
\end{lemma}
\begin{proof}
Using that $X_i$ is a proper subsets of $X$, the chain rule gives
$$
   f(X,y_i\|Y_{i-1}) = f(X,y_i\|X_iY_{i-1}) + f(X_i,y_i\|Y_{i-1}).
$$
Since $X_iy_i\in\mathcal L$ implies $X_iY_{i-1}y_i\in\mathcal L$,
Conditions (i) and (ii) above imply
$$
   f(X,y_i\|X_iY_{i-1}) = f(X,y_i\|aX_iY_{i-1})+f(a) \ge 1,
$$
proving the lemma.
\end{proof}
\begin{proof}[Proof of property (b)]
Since $Y_0=\{y_0\}$ and $Y_i=y_iY_{i-1}$, the chain rule gives
$$
   f(X,Y_i\|y_0)=f(X,y_i\|Y_{i-1})+f(X,y_{i-1}\|Y_{i-2})+\cdots
          +f(X,y_1\|Y_0).
$$
According to Lemma \ref{lemma:3}, each term on the right hand side is at
least $1$. Therefore,
\begin{equation}\label{eq:tight}
   f(X) \ge f(X\|y_0) \ge f(X,Y_{2^k-1}\|y_0) \ge 2^k-1.
\end{equation}
Now $X$ has $k$ elements, and $\sum_{b\in X} f(b) \ge f(X) \ge 2^k-1$. It
follows that there is an element $b\in X$ satisfying
$$
    f(b) \ge \frac{2^k-1}{k} > \frac{n}{2\log_2 n},
$$
since $n=1+k+2^k$. This and $f(a)=1$ implies $C_{f,a}\ge n/(2\log_2 n)$, as
claimed.
\end{proof}
\begin{proof}[Proof of property (c)]
It remains to construct a polymatroid on the base set $\{a\}\cup X\cup Y$
that satisfies conditions (i) and (ii). The construction also shows that the
bound obtained in (\ref{eq:tight}) is tight and cannot be improved. In the
construction each element of $N$ gets one or more independent random bits.
The value of the polymatroid $f$ on $A$ is the Shannon entropy of the
variables in $A$. First, $a\in N$ gets a single, unbiased random bit $r\in
\{0,1\}$, thus $f(a)=1$. The condition $A\in\mathcal L$ means
$f(A)=f(aA)$; this holds if and only if the value of the random bit given to
$a$ is determined by the values of the variables in $A$. Similarly, the
condition $A\in\mathcal U$ means $f(aA)=f(A)+f(a)$, that is, the value of
$a$ is independent of the values of the variables in $A$.

Let us introduce some notation. Enumerate elements of $X$ as
$\{x_1,\dots,\allowbreak x_k\}$. For each $J\subseteq \{1,\dots,k\}$ let
$X_J=\{x_i:i\in J\}$, and $J\restr i=\{j\in J:j<i\}$. For example, $J\restr
1$ is the empty set for every such subset $J$. Subsets of $X$ were arranged
and indexed from $0$ to $2^k-1$; let $\sharp J$ be the index of the subset
$X_J$ in this order; thus $\sharp\{1,\dots,k\}=0$, and $\sharp\mkern1mu
\emptyset=2^k-1$. We write $J\prec K$ when $X_J$ precedes $X_K$, that is,
when $\sharp J<\sharp K$. By the construction, $J\prec K$ implies that $J$
is \emph{not} a subset of $K$, implying that the difference $J \sm K$ is not
empty. Elements of the set $Y$ are written as $y_K$ using subsets of
$\{1,\dots,k\}$ instead of specifying the index $\sharp K$ explicitly as in
$y_{\sharp K}$.

Using this notation, condition (i) requires that (the value of) $a$ is
determined by the values of $X_J$ and $y_J$, while (ii) requires $a$ to be
independent of the values of $X_J$ and $\{y_K: K\prec J\}$ for all subsets
$J$ of $\{1,\dots,k\}$. Let us see the construction.
\begin{itemize}
\item $a$ get a single random bit $r$.
\item $x_i$ for $1\le i\le k$ gets $2^{i-1}$ many independent random
bits, denoted as $r^{(i)}_J$ for every $J\subseteq \{1,\dots,i-1\}$.
\item $y_K$ gets the single bit $y_K=r+\sum_{j\in K} r^{(j)}_{K\restr j}$;
addition is modulo $2$.
\end{itemize}
Clearly, $X_J$ and $y_J$ together determine the value of $r$; moreover
$X$ gets a total of $1+2+\cdots+2^{k-1}=2^k-1$ independent random bits, thus
$f(X)=2^k-1$, achieving the bound in (\ref{eq:tight}). It remains to show
that $r$ is independent of $X_J$ and $\{y_K:K\prec J\}$. Since $K\prec J$
implies that $K\sm J$ is not empty, it suffices to show that $r$ is
independent of $X_J$ and the variables $\{y_K: K\sm J\ne\emptyset\}$. Fixing
the values of these random variables, the bits $r^{(j)}_L$ are fixed for
$j\in J$. The unknown bits are $r$ and $r^{(j)}_L$ for $j\notin J$ so that
their values must provide the correct $y_K$ values. It means that they must
satisfy the system of equations
$$
   r+\sum_{j\in K\sm J} r^{(j)}_{K\restr j} = y_K+\sum_{j\in K\cap J}
r^{(j)}_{K\restr j}, ~~~~~ K\sm J\ne\emptyset,
$$
where the addition is modulo $2$. The right hand side values are fixed.
Since none of the sums on the left hand side are empty, this system clearly
has the same number of solutions for $r=0$ and for $r=1$, which proves the
required independence.
\end{proof}
\let\qed\relax\end{proof}

\section{Conclusions}\label{sec:conclusion}

Motivated by results in secret sharing \cite{beimel}, we define the quantity
$\lambda_n$, which bounds the supremum of the ratio $f(b)/f(a)$ for an
$a$-reduced polymatroid $f$ on an $n$-element ground set. This essentially
characterizes the maximum degree to which the value of a polymatroid at a
specific variable $a$ contrains the values of the remaining singletons. We
show that this maximal value is attained by an extremal polymatroid,
indicating that $\lambda_n$ is deeply tied to the structure of the
submodular cone first studied by J.~Edmonds \cite{edmonds}. Geometrically,
$\lambda_n$ also bounds the maximum elongation of the bounding box $\mathbb
B_f$ of the base polytope of $f$, providing new insights into the geometry
of generalized permutahedra \cite{piecewise}.

Our main result provides lower and upper bounds on $\lambda_n$. The doubly
exponential upper bound follows from a rough estimate of the maximal rank
sub-determinants of the matrix of polymatroid axioms. The resulting upper bound
$2^{2^n-1}$ can be slightly tightened to $2^{2^n-n-2}$ without altering its
asymptotic magnitude. On the other hand, we prove that $\lambda_n$ is at least
$\Omega(n/\log n)$, thus it grows almost linearly. The main idea is that, in
a polymatroid, the inequality
\begin{equation}\label{eq:bb}
     0 \le f(a\|A) \le f(a)
\end{equation}
holds for all subsets $A$ of the ground set $N$. We partition specific
subsets of $N\sm a$ into two families, $\mathcal L$ and $\mathcal U$,
where the lower and upper bounds of (\ref{eq:bb}) are tight, respectively.
If $f(a\|A)$ takes one
of the extremal values, then its $a$-reduced component also takes the same
extremal values. This approach bypasses the need to explicitly verify
whether the construction is $a$-reduced. We construct families
$\mathcal L$ and $\mathcal U$ that force a small subset of $N$ to take a large
value, leading to the stated lower bound. We remark that the applied method
alone cannot produce superlinear lower bound on $\lambda_n$, see
\cite{sslarge}. We conjecture that the polymatroid presented in the proof of
Theorem \ref{thm:main}(c) is actually $a$-reduced, which would provide the
stronger lower bound $n-2\log_2 n\le\lambda_n$.

The complete list of extremal polymatroids is available for $n\le 5$, while,
for $n=6$, a partial list containing around $4.0{\cdot}10^{10}$ extremal
polymatroids has been generated \cite{impossible}. These lists provide
exact values for $n\le 5$ and a lower estimate for $n=6$:
$$
   \lambda_3=1,~~~ \lambda_4=2,~~~  \lambda_5=4, ~~~ \lambda_6\ge 9.
$$
Based on these sporadic values, we conjecture that $\lambda_n$ grows at
least exponentially. Proving this conjecture, along with lowering the doubly
exponential upper bound, remains a challenging open problem.

\section*{Funding}
The research reported in this paper was partially funded by the ERC Advanced
Grant ERMiD.

\end{document}